\documentclass[11pt,reqno]{amsart}
\usepackage{amssymb,latexsym}
\usepackage{graphicx}
\usepackage{fancyhdr,amssymb}
\usepackage{url}		

\theoremstyle{plain}
\numberwithin{equation}{section}
\newtheorem{thm}{Theorem}[section]
\newtheorem{theorem}[thm]{Theorem}
\newtheorem{lemma}[thm]{Lemma}
\newtheorem{example}[thm]{Example}
\newtheorem{definition}[thm]{Definition}
\newtheorem{proposition}[thm]{Proposition}
\newtheorem{corollary}[thm]{Corollary}
\newtheorem{remark}[thm]{Remark}

\begin{document}
\fancyhead{}
\renewcommand{\headrulewidth}{0pt}
\fancyfoot{}
\fancyfoot[LE,RO]{\medskip \thepage}
\fancyfoot[LO]{\medskip MONTH YEAR}
\fancyfoot[RE]{\medskip VOLUME , NUMBER }

\setcounter{page}{1}

\title[Three analogues of Stern's diatomic sequence]{Three Analogues of Stern's diatomic sequence}
\author{Sam Northshield}
\address{Department of Mathematics\\
              SUNY-Plattsburgh\\
                Plattsburgh, NY 12901}
\email{northssw@plattsburgh.edu}

\begin{abstract}  We present three analogues of Stern's diatomic sequence.
\end{abstract}

\maketitle

\section{Introduction}

Stern's diatomic sequence $a_1=1,  a_{2n}=a_n,  a_{2n+1}=a_n+a_{n+1}$
is a particularly well studied sequence (see, e.g., \cite{B1}, \cite{L}, \cite{N} and references therein,   as well as \cite{SW}).    
The first section is devoted to showing that this sequence is \textit{interesting}.    In particular, we shall look at the following properties.   

\begin{itemize}
\item $n\mapsto a_{n+1}/a_n$ is a bijection between the positive natural numbers and the positive rational numbers,
\item $n/2^k\mapsto a_n/a_{n+2^k}$ extends to a continuous strictly increasing function on $[0,1]$  known as ``Conway's box function'' (it's inverse is $?(x)$, Minkowski's question-mark function),
\item  It shares a number of similarities to the Fibonacci sequence;  in particular, it has a Binet type formula.  
\end{itemize}
\vskip .1 in
The remaining three sections are devoted to three analogues of Stern's sequence:

\begin{itemize}
\item  We replace addition by another binary operation;  in particular, we define $b_1=0, b_{2n}=b_n, b_{2n+1}=b_n\oplus b_{n+1}$ where $x\oplus y=x+y+\sqrt{4xy+1}$.  This sequence is related to Stern's sequence and arises from certain sphere packings.    It has apparently not appeared before in the literature. 
\item  We replace the complementary indexing sequences $\{2n\}$ and $\{2n+1\}$ by another pair of complementary sequences;  in particular, let $R_1=1, R_{\alpha(n)}=R_n, R_{\beta(n)}=R_n+R_{n+1}$ where $\alpha(n)=\lfloor n\phi-1/\phi^2\rfloor$, $\beta(n)=\lfloor n\phi^2+\phi\rfloor$ form a specific pair of complementary Beatty sequences.  This sequence has been extensively studied as $R_n$ is the number of ways $n$ can be represented as a sum of distinct Fibonacci numbers.   
\item  The known Binet type formula for Stern's sequence \cite{N} is written in terms of the sequence $s_2(n)$ (:= the number of terms in the binary expansion of $n$).   We  replace $s_2(n)$by $s_F(n)$(:= the number of terms in the Zeckendorf representation of $n$).     This new sequence,  apparently not studied before, is an integer sequence with several interesting (and several conjectural) properties.
\end{itemize}

\vskip .4 in
\section{Stern's Diatomic Sequence}

Consider the following ``diatomic array" \cite{B1} formed as a variant of Pascal's triangle;   each entry is either the value directly above or else the sum of the two above it.
$$\begin{array}{cccccccccccccccccc}1&&&&&&&&&&&&&&&&1\\1&&&&&&&&2&&&&&&&&1\\1&&&&3&&&&2&&&&3&&&&1
\\1&&4&&3&&5&&2&&5&&3&&4&&1\\1&5&4&7&3&8&5&7&2&7&5&8&3&7&4&5&1\\.&.&.&.&.&.&.&.&.&.&.&.&.&.&.&.&.
\end{array}$$
The word ``diatomic" is used here since every entry of the diatomic array gets its value from either one or two entries above and gives that value to three entries below, hence has ``valence" 4 or 5 (hence the diatomic array models a kind of crystalline alloy of two elements).   

Ignoring the right most column and reading the numbers as in a book, we get Stern's diatomic sequence:
$$1,1,2,1,3,2,3,1,4,3,5,2,5,3,4,1,5,...$$
The sequence is thus defined by the recurrence
$$a_1=1,  a_{2n}=a_n,  a_{2n+1}=a_n+a_{n+1}. \eqno{(1)}$$
We define $a_0$ to be 0 (the value consistent with $a_{2\cdot 0+1}=a_0+a_1$).  

Perhaps the most celebrated property of this sequence is that 
every positive rational number is represented exactly once as  $a_{n+1}/a_n$.  See, for example, \cite{CW} or \cite{N}.     We rephrase this fact as a theorem.

\begin{theorem}  Every ordered pair of relatively prime positive integers appears exactly once in the sequence $\{(a_n,a_{n+1})\}$.  \end{theorem}

\begin{proof} 
For an ordered pair, consider the process of subtracting the smallest from largest (stop if equal).    For example,
 $(4,5)\mapsto (4,1)$ and $(7,3)\mapsto (4,3)$.  By the definition of Stern's sequence, 
$$(a_{2n},a_{2n+1}), (a_{2n+1},a_{2n+2})\longmapsto (a_n, a_{n+1}).$$
Every relatively prime pair appears (if not, then there is an ordered pair not on the list with lowest sum.  Apply the process;  the result has lower sum and so is $(a_n, a_{n+1})$ for some $n$ and so the original pair is either $(a_{2n},a_{2n+1})$ or $(a_{2n+1},a_{2n+2})$).   Every relatively prime pair appears 
\textit{exactly} once since, if not, then there exist $m<n$ with $(a_m,a_{m+1})=(a_n,a_{n+1})$ and such that $m$ is as small as possible.    Applying the process to both implies $\lfloor m/2\rfloor=\lfloor n/2\rfloor$ and thus $a_m=a_{m+1}=a_{m+2}$ which is impossible.   \end{proof}

One can then rewrite any sum over relatively prime pairs in terms of Stern's sequence.  As an example, we rephrase the Riemann hypothesis.  First note that $n\longmapsto {a_{2n}/a_{2n+1}}$ is an explicit bijection from ${\Bbb Z}^+$ to ${\Bbb Q}\cap (0,1)$.
Then the Riemann hypothesis is equivalent to $$\displaystyle\sum_{a_{2n+1}<x} e^{2\pi i a_{2n}/a_{2n+1}}=O(x^{1/2+\epsilon}).$$
Briefly why this is so:  the M\"obius function can be written as $\mu(n):=\sum_{1\le k\le n, \gcd(n,k)=1}  e^{2\pi i k/n}$ and so the left hand side is really just Merten's function $M(x):=\sum_{n<x} \mu(n)$.   The connection between Merten's function and the Riemann hypothesis is well-known;   see for example \cite{E}.   

Minkowski's question mark function was introduced in 1904 as an example of a ``singular function" (it is strictly increasing yet its derivative exists and equals 0 almost everywhere).    It is defined in terms of continued fractions:
  $$?(x)=2\displaystyle\sum_{n=1}^{\infty}\dfrac{(-1)^{n+1}}{2^{a_1+a_2+...+a_n}}$$
  where $x=1/(a_1+1/(a_2+1/(a_3+... )))$.    By Lagrange's theorem that states that the continued fraction representation of a quadratic surd must eventually repeat,  it is clear that  $?(x)$ takes quadratic surds to rational numbers.

The function $$f:\dfrac k{2^n}\longmapsto \dfrac{a_k}{a_{2^n+k}}$$
extends to a continuous strictly increasing function on $[0,1]$.
This function is known as ``Conway's box function" and its inverse is Minkowski's question mark function $?(x)$.
\begin{figure}[htbp] 
   \centering
   \includegraphics[width=1.5in]{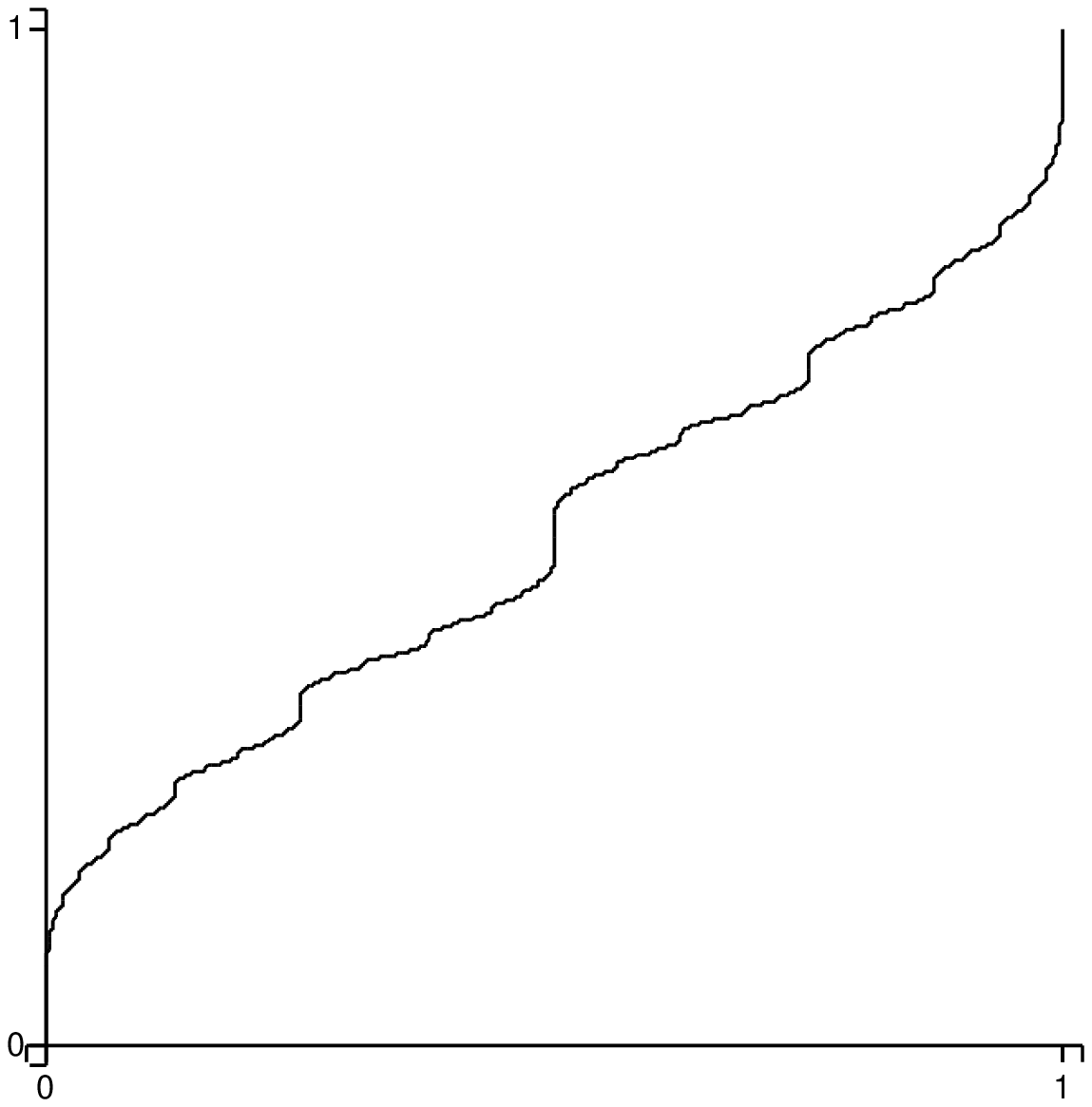} 
    \includegraphics[width=1.5in]{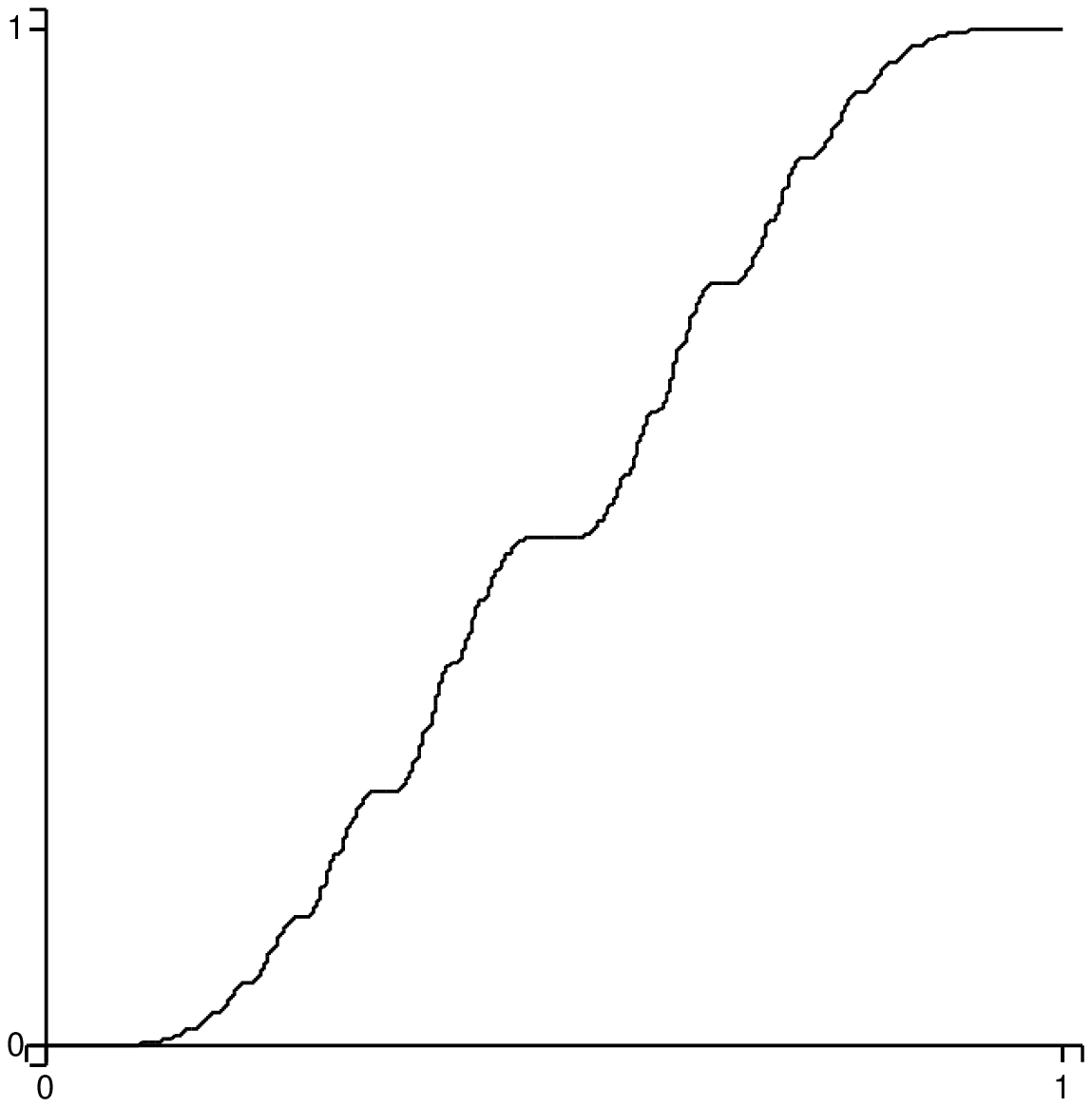} 
   \caption{The graphs of $y=f(x)$ and its inverse $y=?(x)$.}
   \label{fig. 1}
\end{figure}
  See \cite{N} for a proof.

The functions $f(x)$ and $?(x)$ extends to homeomorphisms (or, equivalently, are restrictions of homeomorphisms) between two fractals.     
\begin{figure}[htbp] 
   \centering
   \includegraphics[width=4in]{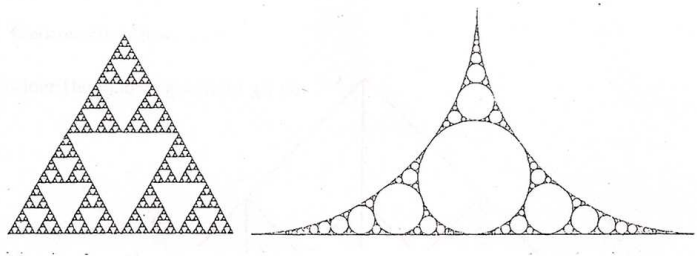} 
   \caption{Sierpinski gasket and an Apollonian circle packing}
   \label{fig. 2}
\end{figure}

 Stern's sequence is related to the Fibonacci sequence in a number of ways.    For example, the Fibonacci sequence is a subsequence:
 $$\begin{array}{cccccccccccccccccc} \framebox 1&&&&&&&&&&&&&&&&1\\1&&&&&&&&\framebox 2&&&&&&&&1\\1&&&&\framebox 3&&&&2&&&&3&&&&1
\\1&&4&&3&&\framebox 5&&2&&5&&3&&4&&1\\1&5&4&7&3&\framebox 8&5&7&2&7&5&8&3&7&4&5&1\\.&.&.&.&.&.&.&.&.&.&.&.&.&.&.&.&.
\end{array}$$
It is easy to see that 
$$a_{J(n)}=F_n \eqno{(2)}$$
where $J(n):=(2^n-(-1)^n)/3$ is the Jacobsthal sequence \cite[A001405]{S}.

 A new result by Coons and Tyler \cite{CT} identifies and proves the asymptotic upper bound:
 $$\limsup_{n\rightarrow\infty}\dfrac{a_n}{(3n)^{\log_2\phi}}=\dfrac1{\sqrt 5}.$$
  The constants involved in this formula are perhaps not so surprising since, by formula (2), it is clear that
   $$\lim_{n\rightarrow\infty}{a_{J(n)}}/({(3{J(n)})^{\log_2\phi}})=1/{\sqrt 5}.$$

Stern's sequence has a few remarkable similarities to the Fibonacci sequence (see \cite{N} and \cite{Nmod2}).  
For example,  Stern's sequence satisfies a modified Fibonacci recurrence: 
 $$a_{n+1}=a_n+a_{n-1}-2(a_{n-1}\mod {a_n}).$$
 Next, certain diagonal sums across Pascal's triangle yield the Fibonacci sequence while the corresponding sums across Pascal's triangle modulo 2 yield Stern's sequence:
 $$\begin{matrix}&&&&1&&&&\\&&&1&&1&&&\\&&1&&2&&\framebox 1&&\\&1&&\framebox3&&3&&1&\\ \framebox1&& 4&&6&&4&&1\\ 
&\cdot&\cdot&\cdot&\cdot&\cdot&\cdot\\ \end{matrix},\begin{matrix}&&&&1&&&&\\&&&1&&1&&&\\&&1&&0&&\framebox 1&&\\&1&&\framebox1&&1&&1&\\ \framebox1&& 0&&0&&0&&1\\ 
&\cdot&\cdot&\cdot&\cdot&\cdot&\cdot\\ \end{matrix}$$
 $$F_{n+1}=\sum_{2i+j=n}{{i+j}\choose i},  \hskip .2 in a_{n+1}=\sum_{2i+j=n} \left[{{i+j}\choose i}\bmod 2\right]$$

Recall
 Binet's formula
  $$F_{n+1}=\dfrac{\phi^{n+1}-\overline\phi^{n+1}}{\phi-\overline\phi}=\displaystyle\sum_{k=0}^n \phi^k\overline\phi^{n-k}.\eqno{(3)}$$
    Stern's sequence satisfies a similar formula:
  $$a_{n+1}=\displaystyle\sum_{k=0}^n \sigma^{s_2(k)}\overline\sigma^{s_2(n-k)}$$
 where $\sigma:=(1+\sqrt{-3})/2$ is a sixth root of unity and $s_2(n)$ is the number of ones in the binary expansion of $n$ 
 \cite[A000120]{S}:
$$\begin{matrix}n&0&1&2&3&4&5&6&7&8...\\
 s_2(n)&0&1&1&2&1&2&2&3&1...\end{matrix}$$
Why this is true:   If 
 $G(x):=\sum \sigma^{s_2(n)}x^n$ and $F(x)$ is the generating function for $\{a_{n+1}\}$, then $G(x)=(1+\sigma x)G(x^2)$
 and $F(x)=(1+x+x^2)F(x^2)$.   For real $x$, since $(1+\sigma x)(1+\overline\sigma x)=1+x+x^2$,  $|G(x)|^2=F(x)$ and the result follows by equating coefficients.

\vskip .4 in
\section{Replacing addition by another operation}

\begin{definition}  For non-negative real numbers $a,b$,   let 
$$\begin{aligned} a\oplus b&=a+b+\sqrt{4ab+1}\\
a\ominus b&=a+b-\sqrt{4ab+1}\end{aligned}.$$\end{definition}

\begin{proposition}  If $a,b,c,d>0$ and $|ad-bc|=1$ then $(ac)\oplus(bd)=(a+b)(c+d)$.   \end{proposition}
\begin{proof}  If $(ad-bc)^2=1$, then $$(ad+bc)^2=1+4abcd$$
and thus  
$$(ac)\oplus (bd)=ac+bd+\sqrt{4abcd+1}=ac+bd+ad+bc.$$
\end{proof}

\begin{remark}  \rm{By the Fibonacci identity
$$F_{n-1}F_{n+1}=F_n^2+(-1)^n,$$
it follows that
$$(F_{n-1}F_n)\oplus(F_nF_{n+1})=(F_{n-1}+F_n)(F_{n}+F_{n+1})=F_{n+1}F_{n+2}.$$
and so the sequence $x_n:=F_nF_{n+1}$ satisfies the modified Fibonacci recurrence
$$x_{n+1}=x_n\oplus x_{n-1}.$$}
\end{remark}
\vskip .2 in

Here we define the first new sequence.  
\begin{definition}  Let $b_1=0$, and for $n\ge 1$,
$$\begin{aligned}  &b_{2n}=b_n\\  &b_{2n+1}=b_n\oplus b_{n+1}.\end{aligned}$$\end{definition}
The sequence begins $$0,0,1,0,2,1,2,0,3,2,6,1,6,2,3,0,4,3,10,2,15,6,12,1,12,6,15,...$$

It is
not immediately clear that this sequence must always be integral.   One way to show this is to express each $b_k$ as a product of elements of Stern's sequence (Theorem 3.6, below).   First we must prove a lemma.

\begin{lemma}  For $m,n\ge 0$, if $m+n=2^j-1$ then $a_{m+1}a_{n+1}-a_ma_n=1$. \end{lemma}

\begin{proof}  We prove this by induction on $j$.   If $m+n=1$, then $a_{m+1}a_{n+1}-a_ma_n=a_1a_2-a_0a_1=1$ and the result holds for $j=1$.   Suppose now that the result holds for a fixed $j$ and that $m+n=2^{j+1}-1$.   Without loss of generality,  $m=2k+1$ and $n=2l$ for some $k,l\ge 0$ (and so $k+l=2^j-1$).   Then
$$\begin{aligned}   a_{m+1}a_{n+1}-a_ma_n&=a_{2k+2}a_{2l+1}-a_{2k+1}a_{2l}\\&=a_{k+1}(a_l+a_{2l+1})-(a_k+a_{k+1})a_{l}=a_{k+1}a_{l+1}-a_ka_l=1\end{aligned}$$ and the result follows.  \end{proof}

\begin{theorem}  If $2^j\le k\le 2^{j+1}$, then
$$b_k=a_{2^{j+1}-k}a_{k-2^j}.$$
\end{theorem}

\begin{proof}  If $k=2^j$ then, because $a_0=0$, $b_k=0=a_{2^{j+1}-k}a_{k-2^j}=a_{2^{j}-k}a_{k-2^{j-1}}$.  

Let $x_k:=a_{2^{j+1}-k}a_{k-2^j}$ where $k\in(2^j, 2^{j+1}).$   Then $2k,2k+1\in(2^{j+1},2^{j+2})$ and thus
$$x_{2k}=a_{2^{j+1}-2k}a_{2k-2^j}=a_{2^j-k}a_{k-2^{j-1}}=x_k$$
and, by lemma 3.5 and proposition 3.2,
$$\begin{aligned} x_{2k+1}&=a_{2^{j+1}-(2k+1)}a_{2k+1-2^j}\\
&=a_{2(2^j-k-1)+1}a_{2(k-2^{j-1})+1}\\
&=(a_{2^j-k-1}+a_{2^j-k})\cdot(a_{k-2^{j-1}}+a_{k+1-2^{j-1}})\\
&=(a_{2^j-k-1}a_{k+1-2^{j-1}})\oplus(a_{2^j-k}a_{k-2^{j-1}})=x_{k+1}\oplus x_k\end{aligned}.$$
Hence $b_k=x_k$ for all $k$, and the result follows.\end{proof}

\begin{corollary}  $b_n\in {\Bbb N}$.  \end{corollary}

As seen in section 2, Stern's diatomic sequence leads to a construction of Conway's box function $f(x)$, the inverse of Minkowski's question-mark function.  
The sequence $\{b_k\}$ gives rise to a similar function that turns out to be closely related to $f(x)$.   

\begin{definition} For $k,n\in {\Bbb N}$,  $k\le 2^n$, let
$$g\left(\frac k{2^n}\right):=\dfrac{b_k}{b_{2^n+k}}.$$ \end{definition}

\begin{theorem}   The function $g(x)$ extends to a continuous function on $[0,1]$ that satisfies, for $x\in (2^{-j-1}, 2^{-j})$,
$$g(x)=f(2^{j+1}x-1)[1-jf(2x)]$$
where $f(x)$ is Conway's box function.   \end{theorem}

\begin{proof}   Let $x=k/2^n$.    Then  $2^{n-j-1}\le k \le 2^{n-j}$ for some $j\ge 0$.    
Since $2^n\le 2^n+k\le 2^{n+1},$ it follows from Theorem 3.6 that
$$b_k=a_{2^{n-j}-k}a_{k-2^{n-j-1}}\text{ and  } b_{2^n+k}=a_{2^n-k}a_k.$$
By \cite[formulas (2) and (3)]{N}, 
$$a_{2^n-k}=ja_k+a_{2^{n-j}-k}$$
and thus
$$\begin{aligned}  g(x)&=g\left(\frac k{2^n}\right)=\dfrac{b_k}{b_{2^n+k}}=\dfrac{a_{2^{n-j}-k}a_{k-2^{n-j-1}}}{a_{2^n-k}a_k}\\
&=\dfrac{(a_{2^n-k}-ja_k)a_{k-2^{n-j-1}}}{a_{2^n-k}a_k}=\dfrac{a_{k-2^{n-j-1}}}{a_k}\left(1-\dfrac{ja_k}{a_{2^n-k}}\right)\\
&=f\left(\dfrac k{2^{n-j-1}}-1\right)\left[1-jf\left(\dfrac k{2^{n-1}}\right)\right]=f(2^{j+1}x-1)[1-jf(2x)].\end{aligned}$$

The extension of $g(x)$ to a continuous function on $[0,1]$ follows from the facts that $f$ extends to a continuous function on $[0,1]$ and
$f(2^{-j})= 1/(j+1).$
\end{proof}

A scaled version of the graph of $y=f(x)$ appears in the graph of $y=g(x)$.

\begin{figure}[htbp] 
   \centering
 \includegraphics[width=1.5in]{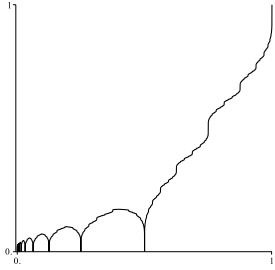}    
 \caption{Singular function associated with $\{b_n\}$}
   \label{fig. 3}
\end{figure}

The restriction of $g(x)$ to $[1/2,1]$ is just a scaled version of $f(x)$:
\begin{corollary}  $g(x)=f(2x-1)$ for $x\in [1/2,1]$.   \end{corollary}

Recall that every positive rational number appears exactly once in the set $\{a_{k+1}/a_k:  k\in{\Bbb N}\}$.    We prove an analogue for the sequence $\{b_k\}$.
We use the expression ``$A=\square$" to mean that $A=n^2$ for some integer $n$.

\begin{theorem}  Every element of $\{(a,b)\in{\Bbb N}^2:  4ab+1=\square\}$ appears exactly once in the sequence $\{(b_k,b_{k+1}): k\in{\Bbb N}\}$.  \end{theorem}

\begin{proof}
Consider the following analogue of the (slow) Euclidean algorithm.

$$M_{\oplus}:  (a,b)\longmapsto\begin{cases} (a,a\ominus b) &\text{if $a<b$, }\\
(a\ominus b, b) &\text{if $b<a$, }\\
\text{stop }&\text{if $a=b$. }\end{cases}$$

Suppose $(a,b)\in{\Bbb N}^2$, with $4ab+1=\square$.   
If $a\ominus b<0$ then it is easy to see that $(a-b)^2<1$ and thus $a=b$.   In this case, since $4a^2+1\neq\square$ unless $a=0$, the only possibility is $a=b=0$.
Hence, $M_{\oplus}((a,b))\in {\Bbb N}^2$ and, if this algorithm terminates at all, it must terminate at $(0,0)$.

With $(a,b)\in{\Bbb N}^2$, with $4ab+1=\square$, let $k:=\sqrt{4ab+1}$.   If $0<a<b$, then $a^2<ak$ and thus
$$a(a\ominus b)=a(a+b-k)=a^2+ab-ak<ab.$$
In general, the product of numbers in $M_{\oplus}((a,b))$ is strictly less than the product $ab$ and thus the algorithm will eventually reach, without loss of generality, $(0,b)$.   If $b=0$ then the algorithm stops.   On the other hand, if $b>0$, it is easy to see that $M_{\oplus}((0,b))=(0,b-1)$, and thus the algorithm will terminate at $(0,0)$.  

Let $B_n:=(b_n,b_{n+1})$.   By the definition of the sequence $\{b_k\}$,  it's easy to see that for $n>1$,
$$M_{\oplus}: B_{2n}, B_{2n+1}\longmapsto B_n$$
and, moreover, if $M_{\oplus}: (a,b)\mapsto B_n$, then either $(a,b)=B_{2n}$ or $(a,b)=B_{2n+1}$.   

If  $(a,b)\in{\Bbb N}^2$, with $4ab+1=\square$ is not of the form $B_n$ for some $n$, then all of its successors under  $M_{\oplus}$, including $(0,0)$, are not either -- a contradiction.   Hence every $(a,b)\in{\Bbb N}^2$, with $4ab+1=\square$
is of the form $B_n$ for some $n$.  

The pair $(0,0)$ appears only once and, in general, no pair appears more than once in $\{ B_n\}$ for, otherwise, there exists a smallest $n>1$ such that $B_n=B_m$ for some $m>n$.   Applying $M_{\oplus}$ to both $B_m$ and $B_n$ forces
$\lfloor n/2\rfloor=\lfloor m/2\rfloor$ and therefore $m=n+1$.   Thus $b_n=b_{n+1}=b_{n+2}$, a contradiction.  
\end{proof}

A generalization of $\oplus$ is as follows:  For a given number $N$, define
$$x\underset{N}{\oplus} y:=x+y+\sqrt{4xy+N}.$$

\begin{remark}  $a,b,a\underset{N}{\oplus}b$ solve 
$$2(x^2+y^2+z^2)-(x+y+z)^2=N.$$

Defining $\underset{N}{\ominus}$ in the obvious manner,  
$$(a\underset{N}{\oplus}b)\underset{N}{\ominus}b=a.$$

Every non-zero complex number $z$ can be represented uniquely as $re^{i\theta}$ for some positive $r$ and some $\theta\in[0,2\pi)$ and so we define $\sqrt z:= \sqrt r e^{i\theta/2}$.    Hence $\underset{N}{\oplus}$ and $\underset{N}{\ominus}$
are well defined for complex $N$.
\end{remark}

We may then generalize $\{b_k\}$.
\begin{definition}  Given a (complex) number $A$, let $c_1=c_2=A$ and, for $n\ge 1$, 
$$\begin{aligned}  &c_{2n}=c_n\\  &c_{2n+1}=c_n\underset{N}{\oplus} c_{n+1}.\end{aligned}$$\end{definition}
It turns out that such a sequence can be expressed as a linear combination of the sequences $\{a_k^2\}$ and $\{b_k\}$. 
We first need a lemma. 

\begin{lemma}  For $k\ge 1$,  
$$a_k^2b_{k+1}+a_{k+1}^2b_k+1=a_ka_{k+1}\sqrt{4b_kb_{k+1}+1}.$$\end{lemma}

\begin{proof}  Let $s_k:=\sqrt{4b_kb_{k+1}+1}$.   Note that
$$b_{2k+1}=b_k+b_{k+1}+s_k.$$
Then
$$\begin{aligned}s_{2k}^2&=4b_{2k}b_{2k+1}+1=4b_k(b_k+b_{k+1}+s_k)+1\\
&=4b_k^2+s_k^2+4b_ks_k=(2b_k+s_k)^2\end{aligned}$$ and so 
$$s_{2k}=2b_k+s_k.$$   Similarly, 
$$\begin{aligned}s_{2k+1}^2&=4b_{2k+1}b_{2k+2}+1=4b_{k+1}(b_k+b_{k+1}+s_k)+1\\
&=4b_{k+1}^2+s_k^2+4b_{k+1}s_k=(2b_{k+1}+s_k)^2\end{aligned}$$
and so $$s_{2k+1}=2b_{k+1}+s_k.$$   

Note that 
$$a_1^2b_2+a_2^2b_1+1=1=a_1a_2\sqrt{4b_1b_2+1}$$ 
and so the lemma holds for $k=1$.    

Suppose the lemma holds for a particular $k$. 
We show it works for $2k$ and $2k+1$ and thus, by induction, the lemma will be shown.
$$\begin{aligned} a_{2k}^2b_{2k+1}&a_{2k+1}^2b_{2k}+1=a_k^2(b_k+b_{k+1}+s_k)+(a_k+a_{k+1})^2b_k+1\\
&=a_{k}^2b_{k}+a_{k}^2b_{k+1}+a_k^2s_k+a_k^2b_k+2a_ka_{k+1}b_k+a_{k+1}^2b_{k}+1\\
&=a_k^2(2b_k+s_k)+2a_ka_{k+1}b_k+a_k^2b_{k+1}+a_{k+1}^2b_{k}+1\\
&=a_k^2(2b_k+s_k)+2a_ka_{k+1}b_k+a_ka+{k+1}s_k\\
&=a_k(a_k+a_{k+1})(2b_k+s_k)=a_{2k}a_{2k+1}s_{2k}\end{aligned}$$
and thus the lemma works for $2k$.

$$\begin{aligned} a_{2k+1}^2b_{2k+2}&a_{2k+2}^2b_{2k+1}+1=(a_k+a_{k+1})^2b_{k+1}+a_{k+1}^2(b_k+b_{k+1}+s_k)+1\\
&=a_{k}^2b_{k+1}+2a_ka_{k+1}b_{k+1}+a_{k+1}^2b_{k+1}+a_{k+1}^2b_{k}+a_{k+1}^2b_{k+1}+a_{k+1}^2s_{k}+1\\
&=a_k^2(2b_{k+1}+s_k)+a_k^2b_{k+1}+a_{k+1}^2b_{k}+1+2a_ka_{k+1}b_{k+1}\\
&=a_k^2(2b_{k+1}+s_k)+a_ka_{k+1}s_k+2a_ka_{k+1}b_{k+1}\\
&=a_{k+1}(a_k+a_{k+1})(2b_{k+1}+s_k)=a_{2k+2}a_{2k+1}s_{2k+1}\end{aligned}$$
and thus the lemma works for $2k+1$.
\end{proof}

\begin{theorem}  Given $A,B$, let $c_k:=Aa_k^2+Bb_k.$
Then $\{c_k\}$ has $c_1=c_2=A$ and, for $N=4AB+B^2$,
$$\begin{aligned}  &c_{2n}=c_n\\  &c_{2n+1}=c_n\underset{N}{\oplus} c_{n+1}.\end{aligned}$$
 \end{theorem}

\begin{proof} 
$$\begin{aligned} c_kc_{k+1}+AB&=(Aa_k^2+Bb_k)(Aa_{k+1}^2+Bb_{k+1})+AB\\
&=A^2a_k^2a_{k+1}^2+B^2b_kb_{k+1}+AB(a_{k+1}^2b_k+a_k^2b_{k+1}+1)\\
&=A^2a_k^2a_{k+1}^2+B^2b_kb_{k+1}+ABa_ka_{k+1}\sqrt{4b_kb_{k+1}+1}\end{aligned}$$
 and so
 $$\begin{aligned}4c_kc_{k+1}+N&=4A^2a_k^2a_{k+1}^2+4B^2b_kb_{k+1}+B^2+4ABa_ka_{k+1}\sqrt{4b_kb_{k+1}+1}\\
 &=(2Aa_ka_{k+1}+B\sqrt{4b_kb_{k+1}+1})^2\end{aligned}$$
 and thus
 $$\begin{aligned} c_k\underset{N}{\oplus} c_{k+1}&=(Aa_k^2+Bb_k)+(Aa_{k+1}^2+Bb_{k+1})+\sqrt{4c_kc_{k+1}+N}\\
 &=Aa_k^2+Bb_k+Aa_{k+1}^2+Bb_{k+1}+2Aa_ka_{k+1}+B\sqrt{4b_kb_{k+1}+1}\\
 &=A(a_k+a_{k+1})^2+B(b_k+b_{k+1}+\sqrt{4b_kb_{k+1}+1}\\
 &=Aa_{2k+1}^2+Bb_{2k+1}=c_{2k+1}.\end{aligned}$$
 
 Since $$c_{2k}=Aa_{2k}^2+Bb_{2k}=Aa_k^2+Bb_k=c_k,$$
 the theorem is shown. \end{proof}

\begin{example}  Let $N=-3$, $c_1=c_2=1$,  we see that $A=1$, $B=-1$, and thus $c_k=a_k^2-b_k$.   \end{example}

 If to every local cut point $P$ in the fractal CP appearing in figure 2 one attaches a sphere above but tangent to the plane at that point with curvature (1/radius) equal to the sum of the curvatures of the two circles meeting there, then one gets a 3-dimensional generalization of Ford circles.    The curvatures (similarly, the product of local cut points and corresponding curvatures) along any circular arc are from a sequence $\{c_n\}$ for appropriately chosen $N$  (see \cite{N2} and references therein for a discussion of various types of ``Ford spheres").

Consider the sequence $\{b_k\}$ written in tabular form:
$$\begin{matrix} 0&&&&&&&&\\0&1&&&&&&&\\0&2&1&2&&&&&\\0&3&2&6&1&6&2&3&\\0&4&3&10&2&15&6&12&...\\
.&.&.&.&.&.&.&.&.\end{matrix}$$
It is apparent that every column is an arithmetic sequence and, moreover, the defining differences are respectively
$$0,1,1,4,1,9,4,9,...,$$
the squares of Stern's diatomic sequence $\{a_k^2\}$.  
This is, in fact, true.   We shall express this result as a formula.

\begin{theorem} For $0\le k<2^j$,  
$$b_{2^{j+1}+k}=a_k^2+b_{2^j+k}.$$\end{theorem}

\begin{proof} Assume  $0\le k<2^j$.
Since $2^j\le 2^j+k<2^{j+1}$,   Theorem 3.6 implies
$$b_{2^{j+1}+k}=b_{2^{j}+2^j+k}=a_{2^{j+1}-(2^j+k)}a_{2^j+k-2^j}=a_{2^j-k}a_k.$$
By  \cite[formulas (2) and (3)]{N} ,
$$a_{2^{j+1}-k}=a_k+a_{2^{j}-k}$$ 
and thus
$$b_{2^{j+2}+k}=a_{2^{j+1}-k}a_k=(a_k+a^{2^j-k})a_k=a_k^2+a_{2^j-k}a_k=a_k^2+b_{2^{j+1}+k}.$$
The result follows by induction.\end{proof}

\begin{remark}$\{b_{2k-1}\}$  appears as \cite[A119272]{S},  the product of numerators and denominators in the Stern-Brocot tree. \end{remark}

\begin{remark} For a fixed $(x,y)$,
$z=x\oplus y$ and $z=x\ominus y$ are the two solutions of 
$$2(x^2+y^2+z^2)-(x+y+z)^2=1.$$ \end{remark}


\vskip .4 in
\section{Fibonacci representations}

A \textit{Fibonacci representation} of a number $n$ is a way of writing that number as a sum of distinct Fibonacci numbers.   
One such representation is, of course, the Zeckendorf representation which is gotten by the greedy algorithm and which is characterized by having no two consecutive Fibonacci numbers.    In general, a given $n$ has several Fibonacci representations,   the number of such we call $R_n$.      The sequence $\{R_n\}$ is extremely well studied;   see papers by Klarner \cite{K},  Bicknell-Johnson \cite{B1,B2}, and Stockmeyer \cite{St}, for example.  

A string of 0s and 1s is a finite word with alphabet $\{0,1\}$   (equivalently, an element of $\{0,1\}^*$).  Often we denote such a word by $\omega$.  
We shall think of such strings as Fibonacci representations:   we shall assign a numerical value $[\omega]$ to a string $\omega$ by the formula
$$[i_1i_2...i_k ]= \sum i_j F_{k+2-j}.$$
For example, $[0100]=[0011]=3$ and  $[01010011]=21+8+2+1=32$.

The generating function for $\{R_n\}$ has an obvious product formulation.
\begin{proposition}  The sequence $(R_n)$ satisfies
$$\sum_{n=0}^{\infty} R_nx^n= \prod_{i=2}^{\infty}\left(1+x^{F_i}\right)$$
where $F_n$ denotes the $n$th Fibonacci number.  \end{proposition}

Next, we define the \textit{Fibonacci shift}:  $$\rho(n):=\lfloor n\phi +1/\phi\rfloor$$
that satisfies $\rho([\omega])=[\omega 0]$ for every string $\omega$.    This shift has been studied before;   for example, it appears in \cite[graffiti, p. 301]{GKP}.

 \begin{theorem}   For $c_i\in\{0,1\}$,  $i=2,...,N$,  
$$\rho\left(\sum_{i=2}^N c_iF_i\right) =\sum_{i=2}^N c_iF_{i+1}.$$\end{theorem}

\begin{proof}  By Binet's formula (3), 
$$\phi F_n=F_{n+1}-\overline\phi^n.$$
For any choice $c_i\in\{0,1\}$ for $i=2,...,N$, note that
$$-1/\phi^2=\sum_{n=1}^{\infty}\overline\phi^{2n+1}<\sum_{i=2}^Nc_i\overline\phi^i<\sum_{n=1}^{\infty}\overline\phi^{2n}=1/\phi$$
and therefore 
$$0<-\sum_{i=2}^Nc_i\overline\phi^i-\overline\phi<1.$$
Hence,
$$\begin{aligned}  \rho\left(\sum_{i=2}^Nc_iF_i\right)&=\left\lfloor\phi\sum_{i=2}^Nc_iF_i-\overline\phi\right\rfloor\\
&=\sum_{i=2}^N c_iF_{i+1}+\left\lfloor -\sum_{i=2}^Nc_i\overline\phi^i-\overline\phi\right\rfloor=\sum_{i=2}^N c_iF_{i+1}. \end{aligned}$$
\end{proof}

In terms of $\rho(n)$, we may define $\{R_n\}$ recursively.  
Clearly, $R_0=R_1=1$.   A representation of $n$ either ends in $0$ in which case $n=[\omega 0]$ where $\rho([\omega])=n$ or else it ends in $1$ in which case $n=[\omega 1]$ and so $n-1=[\omega 0]=\rho([\omega])$.    Hence, for all $n\ge 1$,
$$R_n:=\sum_{\rho(i)\in\{n,n-1\}} R_i.$$

Note that the function $\rho_2(n):=\rho(\rho(n))=\lfloor n\phi^2+1/\phi\rfloor$ is an example of a Beatty sequence (i.e., of the form $\lfloor an+b\rfloor$) and so has a complementary Beatty sequence, namely $T(n):=\lfloor n\phi+2/\phi\rfloor$.
For example,  $$\rho_2(n)=0, 3, 5, 8, 11, 13, 16, 18, 21, 24,...$$ and $$T(n)=1,2, 4,6,7, 9,10,12,14,15,...$$  
The following characterization could be used as a new definition of $\{R_n\}$.

\begin{theorem}  For $n\ge 1$, and $T(n):=\lfloor n\phi+2/\phi\rfloor$,
 $$R_{\rho_2(n)}=R_n+R_{n-1}$$
 and
 $$R_{T(n)}=R_n.$$\end{theorem}

\begin{proof}   
Since $\phi\in(1,2)$, 
$\rho(n)\in\{\rho(n+1)-1,\rho(n+1)-2\}.$
Since $2\phi>3$, $(n-1)\phi+1/\phi\le (n+1)\phi+1/\phi-3$
and so $\rho(n-1)<\rho(n+1)-2.$
Note that $$T(n)=\lfloor n\phi+2/\phi\rfloor=\lfloor(n+1)\phi+1/\phi\rfloor-1=\rho(n+1)-1$$
and therefore
$$R_{T(n)}=\sum_{\rho(i)\in\{\rho(n+1)-1,\rho(n+1)-2\}} R_i=R_n.$$

We show the first equation in the theorem by a counting argument.  
By the definition of $\rho(n)$, $\rho_2(n)=\rho(n)+n$ and so 
$$\rho_2(n+1)-\rho_2(n)=\rho(n+1)-\rho(n)+1\in\{2,3\}.$$
For a given $n$,  if $n=[\omega]$ then 
$\rho(\rho(n))=[\omega00]$ and 
$\rho(\rho(n+1))$ equals either $[\omega10]$ or $[\omega11]$.  

Suppose $\rho_2(n+1)-\rho_2(n)=2$.  
The map $\omega\mapsto \omega00$ is a bijection from representations of $n$ to the representations of $\rho_2(n)$ ending in 00 while the
map $\omega\mapsto\omega10$ is a bijection from representations of $n-1$ to the representations of $\rho_2(n)$ not ending in 00.  Hence the first equation holds.   

A similar argument holds when $\rho_2(n+1)-\rho_2(n)=3$.
\end{proof}

\begin{remark}  The sequence $\{R_n\}$ is thus analogous to the alternative form of Stern's sequence:
$$a_{2n}=a_n,  a_{2n-1}=a_n+a_{n-1}.$$\end{remark}

For every word $\omega:=\omega_0\omega_1...\omega_n\in\{0,1\}^*$,  we let $|\omega|:=n+1$ denote the length of $\omega$ and define a point in the complex plane
$$P(\omega):=\sum_{k=0}^n \phi^{-k}(2\omega_k-1-i).$$
We form a graph ${\bf G}$ by putting an edge between $P(\omega)$ and $P(\omega j)$ for $j=0,1$,   $\omega\in\{0,1\}^*$.   This graph is illustrated in Figure 4 below.  
 \begin{figure}[htbp] 
   \centering
   \includegraphics[angle=0, width=4 in]{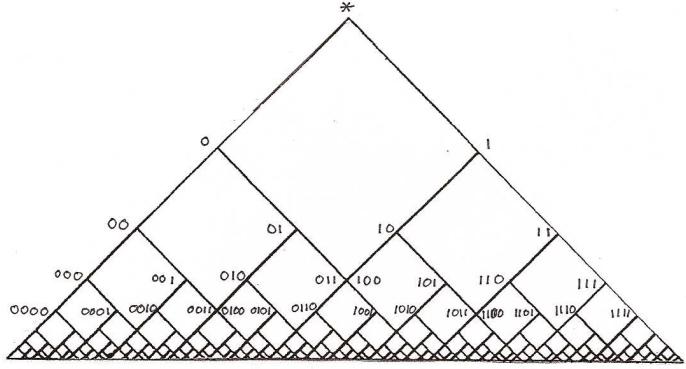} 
   \caption{Fibonacci Representation Graph with words in $\{0,1\}^*$.}
   \label{fig. 4}
\end{figure}
Note further that $P(\omega)=P(\omega')$ iff  $|\omega|=|\omega'|$ and $[\omega]=[\omega']$.   Hence, we may consistently assign the integer $[\omega]$ to each vertex $P(\omega)$ of the graph.    This shows that $R_{[\omega]}$ is the number of downward paths from $P(*)$ to $P(\omega)$ and the graph can be thought of as a kind of hyperbolic Pascal's triangle.    In fact, the portion between 0,01,010,0101,... and 1,10,101,1010,... is really just the ``Fibonacci diatomic array'' appearing in \cite{B2}.  

For $v$ a vertex of the Fibonacci representation graph, let $[v]$ be the number of downward paths from the top vertex to $v$.  

\begin{lemma}  Along the $n$th row of the graph ${\bf G}$, the function $[v]$ forms an increasing sequence of consecutive integers $0,\dots, F_{n+2}-2$.  \end{lemma}

\begin{proof}  Iterates of $\rho(n)+1$ starting at 0 yields the sequence $0,1,3,6,11,..., F_{n+2}-2,...$ (provable  by induction).   Hence the last value of $[v]$ in the $n$th row is $F_{n+2}-2$.     Since $\rho(n+1)-\rho(n)\in\{1,2\}$, the lemma follows.  \end{proof}

A consequence is the following surprising formula:
$$\rho(\rho(\rho(n)+1)))=\rho(\rho(\rho(n))+1)+1.$$

 This graph has numerous interesting properties:
 \begin{itemize}
 \item  Every quadrilateral in the closure of the graph is either a square or a golden rectangle.
  \item All the squares (actually hexagons) are congruent in hyperbolic space with area $\ln\phi$ (and, as hexagons, each edge has length $\ln\phi$).  The figure is thus an aperiodic tiling of part of the upper half-plane ${\bf H}$ (and can be extended to all of ${\bf H}$ ) where all the tiles are congruent!  
  \item The points along any row, when embedded in ${\Bbb  R}$ form part of a one-dimensional quasicrystal.   The lengths of the segments, appropriately scaled, form a word:  $\phi, 1,  \phi, \phi, 1, \phi, ...$, the  ``Fibonacci word''.   
 \item The vertices form a quasicrystal in ${\bf H}$ .  
 \item The graph is the Cayley graph of the ``Fibonacci monoid"  $\langle a,b| abb=baa\rangle$.  
  \item The graph can be constructed by the following recursive procedure starting with a single vertex;   from each of the latest generation of vertices, draw two edges going southeast and southwest respectively,  connect if a hexagon can be formed.   Repeat. 
 \end{itemize}

Something new with respect to the study of $\{R_n\}$ is the development of an analog of Conway's box function.
For $k<F_{n-1}$, define
$$q(k,F_n):=R_k/R_{F_n+k}.$$

\begin{lemma}  For $k=0,..., F_{n-1}-1$,
$$q(T(k),F_{n+1})=q(k,F_n)$$
and
$$q(\rho_2(k),F_{n+2})=q(k,F_n)*q(k-1,F_n)$$
where $*$ denotes ``mediant addition".  \end{lemma}

\begin{proof}  Note that 
$$T(n)=\rho(n+1)-1$$
and so, if $k\le F_{n-1}-1$,
$$T(F_n+k)=F_{n+1}+T(k).$$
Then $$\begin{aligned} q(k, F_n)&=\dfrac{R_k}{R_{F_n+k}}=\dfrac{R_{T(k)}}{R_{T(F_n+k)}}\\&=\dfrac{R_{T(k)}}{R_{F_{n+1}+T(k)}}=q(T(k),F_{n+1})\end{aligned}$$ and the first equation follows.   
Similarly,   
$$\begin{aligned} q(k, F_n)*q(k-1,F_n)&=\dfrac{R_k}{R_{F_n+k}}*\dfrac{R_{k-1}}{R_{F_n+k-1}}=\dfrac{R_k+R_{k-1}}{R_{F_n+k}+R_{F_n+k-1}}\\&=\dfrac{R_{\rho_2(k)}}{R_{\rho_2(F_n+k)}}=\dfrac{R_{\rho_2(k)}}{R_{F_{n+2}+\rho_2(k)}}\\&=q(\rho_2(k), F_{n+2})\end{aligned}$$
and the second equation follows.
\end{proof}
 
As a consequence, if,  as $n\rightarrow\infty$, $k/F_n$ converges to $x\in[0,1/\phi]$, then $q(k,F_n)$ converges to some value, say $Q(x)$.  The function $Q: [0,1/\phi]\rightarrow[0,1]$ is  increasing and continuous.   

\begin{figure}[htbp] 
   \centering
   \includegraphics[angle=0, width=3 in]{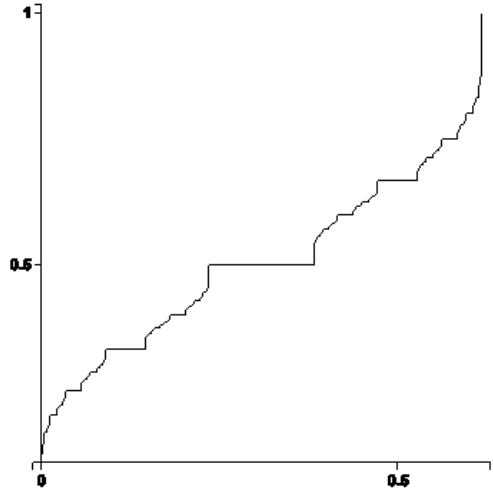} 
   \caption{Analogue of Conway's box function}
   \label{fig. 5}
\end{figure}

Note, however, it is not strictly increasing.   

\begin{lemma}
For $j=0,..., F_{n-1}-1$,  
$$R_{F_{n+2}+j}=R_{F_n+j}+R_j.$$
\end{lemma}

\begin{theorem}
The inverse of $Q$ satisfies, on its irrational points of continuity,  
$$Q^{-1}(x)=\sum_{k=1}^{\infty}\dfrac{(-1)^{k+1}}{\phi^{2(c_1+c_2+...+c_k)-1}}$$
where $x$ has continued fraction decomposition $x=1/(c_1+1/(c_2+1/(c_3+...)))$.  \end{theorem}

\begin{proof}  Recall that $R_{F_m+k}=R_{F_{m+1}-k}$ and so
$$\dfrac1{n+q(k,F_m)}=\dfrac{R_{F_{m+1}-k}}{R_k+nR_{F_{m+1}-k}}=q(F_{m+1}-k,F_{m+2n}).$$
Letting $k/F_m\rightarrow x$ where $x$ is a point of continuity of $B$, we see that 
$$\dfrac1{n+Q(x)}=Q\left(\dfrac{\phi-x}{\phi^{2n}}\right).$$
We may then rewrite:
$$\dfrac{\phi-Q^{-1}(x)}{\phi^{2n}}=Q^{-1}\left(\dfrac1{n+x}\right)$$
and the theorem follows.  \end{proof}
The function $Q(x)$ extends past $1/\phi$ but is no longer monotonic.   

\begin{figure}[htbp] 
   \centering
   \includegraphics[angle=0, width=3 in]{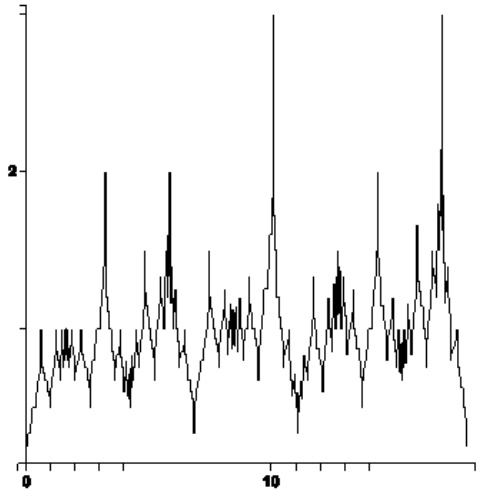} 
   \caption{Analogue of Conway's box function;  larger domain}
   \label{fig. 6}
\end{figure}

Patterns can be found by looking at the ``crushed array'' which is found by stacking rows of terms $R_{F_{n}-1},..., R_{F_{n+1}-2}$ sliding terms to the left on rows:
$$\begin{array}{ccccccccccccccccc} 1&&&&&&&&&&&&&&&&\\1&2&&&&&&&&&&&&&&&\\1&2&2&&&&&&&&&&&&&&
\\1&3&2&2&3&&&&&&&&&&&&\\1&3&3&2&4&2&3&3&&&&&&&&\\1&4&3&3&5&2&4&4&2&5&3&3&4&&&&\\.&.&.&.&.&.&.&.&.&.&.&.&.&.&.&.&.
\end{array}$$
The $k$th column satisfies: $x_{n+2}=x_n+c$ with common difference $c=R_k$  ($R_0=0$).

Alternatively, $x_{n+1}=x_n+x_{n-1}-x_{n-2}$ (a ``dying rabbit" sequence).

$$x_{n+1}=x_n+x_{n-1}-x_{n-2}$$
Characteristic polynomial factors $x^3-x^2-x+1=(x-1)^2(x+1)$
so every example is of the form $x_n=a+bn+c(-1)^n$.  Hence,   $\{x_{2n}\}$ and $\{x_{2n+1}\}$ are arithmetic sequences.
$$x_{n+1}=x_n+x_{n-1}-x_{n-3}$$  e.g., \cite[A023434]{S}
$x^4-x^3-x^2+1=(x-1)(x^3-x-1),$
so every example is of the form $a+br_1^n+cr_2^n+dr_3^n$ where 
$r_1$ is the ``plastic constant", 1.32471795..., the smallest Pisot number, and  $r_2, r_3$ are its algebraic conjugates.   
Such examples are always a constant plus a Padovan sequence $y_{n+1}=y_{n-1}+y_{n-2}$.  E.g., \cite[A000931]{S}
$$x_{n+1}=x_n+x_{n-1}-x_{n-1},$$
is always a constant sequence.

\vskip .4 in
\section{Extending Binet's formula}

Let $s_F(n)$ be the number of terms in the Zeckendorf representation of $n$ (e.g., $s_F(27)=3$).   Equivalently, $s_F(n)$ is the least number of Fibonacci numbers that sum to $n$.    This sequence, for $n=0,1,...$, is 
[A007895] and starts 
$$0,1,1,1,2,1,2,2,1,2,2,2,3,1,2,2,2,3,2,3,3,...$$
Using notation of the previous section, we see that $s_F(n)$ satisfies the recursion:
$$s_F([\omega 0])=s_F([\omega]),  s_F([\omega 01])=s_F([\omega])+1$$
which translates to
$$s_F(\rho(n))=s_F(n), s_F(\rho_2(n)+1)=s_F(n)+1$$
where $\rho(n)$ is the ``Fibonacci shift" defined in Section 4 (just after Proposition 4.1).  
The crushed array for this sequence is
$$\begin{array}{ccccccccccccccccc} 1&&&&&&&&&&&&&&&&\\1&&&&&&&&&&&&&&&&\\1&2&&&&&&&&&&&&&&&\\1&2&2&&&&&&&&&&&&&&
\\1&2&2&2&3&&&&&&&&&&&&\\1&2&2&2&3&2&3&3&&&&&&&&\\1&2&2&2&3&2&3&3&2&3&3&3&4&&&\\.&.&.&.&.&.&.&.&.&.&.&.&.&.&.&.&.
\end{array}$$
Note that columns are constant and that the limiting row is $s_F(n)+1$.   

Replacing $s_2(n)$ by $s_F(n)$ in Binet's formula for Stern's sequence yields our third variant of Stern's sequence:
$$c_{n+1}=\displaystyle\sum_{k=0}^n \sigma^{s_F(k)}\overline\sigma^{s_F(n-k)}.$$
The sequence starts, for $n=1,2,...$, 
$$1,1,2,3,2,4,3,3,6,4,6,6,4,8,6,7,10,6,9,7,5,11,8,....$$
It is always integral since $c_{n+1}$ is an algebraic integer in ${\Bbb Z}[\sigma]$ invariant under complex conjugation.  

A crushed array for this sequence is:
$$\begin{array}{ccccccccccccccccc} 1&&&&&&&&&&&&&&&&\\1&&&&&&&&&&&&&&&&\\2&3&&&&&&&&&&&&&&&\\2&4&3&&&&&&&&&&&&&&\\3&6&4&6&6&&&&&&&&&&&&
\\4&8&6&7&10&6&9&7&&&&&&&&&\\5&11&8&11&13&8&14&10&9&15&9&13&11&&&&\\7&15&11&15&19&12&19&14&11&21&14&19&19&&&&\end{array}$$
The first column,  $x_n:=\{c_{F_n}\}$ apparently satisfies the Padovan recurrence:  $x_{n+2}=x_n+x_{n-1}$.  Moreover, every column is apparently a ``dying rabbit" sequence:  $x_{n+1}=x_n+x_{n-1}-x_{n-3}$ or, more precisely, 
if $x_n:=c_{F_n+k}+c_k$, then $x_{n+2}=x_n+x_{n-1}$.     This is indeed the case which we now prove. 

\begin{theorem}  For $k\le F_{n-2}$, $c_{F_{n+2}+k}=c_{F_{n}+k}+c_k+c_{F_{n-1}+k}$. \end{theorem}

\begin{proof}    
Given a  string $X$ of integers, let $\overline X$ denote the reverse of string $X$, let $X^+$ denote the string $X$ with  1 added to every integer, and let $X^-$ denote the string $X$ with  1 subtracted from every integer (e.g., if $X=1223$, then $\overline X=3221$,  $X^+= 2334$, and $\overline X^-=2110$).  If $X:=t_0...t_{k-1}$, let $G(X):=\sum_{j=0}^{k-1} \sigma^{t_j}.$  Finally, given strings $X$, $Y$,  we let $XY$ denote the concatenation of the two strings and $X-Y$ denote the pointwise difference (e.g.,   if $X=457$ and $Y=123$ then $XY=457123$ and $X-Y=334$).

Let $s_n:=s_F(n)$ be the number of terms in the Zeckendorf representation of $n$.   For any interval $I$, let $s_I$ denote the \textit{string} $s_{i_1}s_{i_2}...s_{i_k}$ where $i_1<i_2<...<i_k$ and $\{i_1, i_2, ... , i_k\}=I\cap {\Bbb N}$.
Then $c_{n}=G((s_{[0,n)}-\overline s_{[0,n)}))$  where the difference of two strings is the string of differences.   Since we will use this formula, we let $\Delta_I=s_I-\overline s_I$  so that $c_n=G(\Delta_{[0,n)})$.  

By the definition of $s$, it's clear that $s_{[F_n, F_n+k)}=s_{[0,k)}^+$ if $k\le F_{n-1}$.    Since
$$s_{[0,F_{n}+k)}=s_{[0,k)}s_{[k,F_{n})}s_{[F_n,F_n+k)}=s_{[0,k)}s_{[k,F_{n})}s_{[0,k)}^+,$$
it follows that 
$$\overline s_{[0,F_n+k)}=\overline s_{[0,k)}^+ \overline s_{[k,F_{n})}\overline s_{[0,k)},$$
and thus
$$\Delta_{[0,F_n+k)}=\Delta_{[0,k)}^- \Delta_{[k,F_{n})}\Delta_{[0,k)}^+.$$
Hence, because  $\sigma^{-1}+\sigma=1$, 
$$c_{F_n+k}=\sigma^{-1} c_k +G(\Delta_{[k,F_{n})}) +\sigma c_k = c_k+ G(\Delta_{[k,F_{n})}).\eqno{(*)}$$

Assuming  $k\le F_{n-2}$,  we see that
$$\begin{aligned}  s_{[0,F_{n+2}+k)}&=s_{[0,k)}s_{[k,F_{n})}s_{[F_n,F_{n}+k)}s_{[F_n+k,F_{n+1})}s_{[F_{n+1},F_{n+1}+k)}s_{[F_{n+1}+k,F_{n+2})}s_{[F_{n+2},F_{n+2}+k)}\\
&=s_{[0,k)}s_{[k,F_{n})}s_{[0,k)}^+s_{[k,F_{n-1})}^+s_{[0,k)}^+s_{[k,F_{n})}^+s_{[0,k)}^+\end{aligned}$$
and thus
$$\overline s_{[0,F_{n+2}+k)}=\overline s_{[0,k)}^+\overline s_{[k,F_{n})}^+\overline s_{[0,k)}^+\overline s_{[k,F_{n-1})}^+\overline s_{[0,k)}^+\overline s_{[k,F_{n})}\overline s_{[0,k)}.$$
Hence, 
$$\Delta_{[0,F_{n+2}+k)}=\Delta_{[0,k)}^-\Delta_{[k,F_{n})}^-\Delta_{[0,k)}\Delta_{[k,F_{n-1})}\Delta_{[0,k)}\Delta_{[k,F_{n}+k)}^+\Delta_{[0,k)}^+.$$
Applying $G$:
$$c_{F_{n+2}+k}=\sigma^{-1} c_k + \sigma^{-1} G(\Delta_{[k,F_{n})}) +c_k + G(\Delta_{[k,F_{n-1})}) +c_k+\sigma G(\Delta_{[k,F_{n})}) +\sigma c_k.$$
Again, since $\sigma^{-1}+\sigma=1$,  and by (*), we have
$$c_{F_{n+2}+k}=3c_k +G(\Delta_{[k,F_{n})}) + G(\Delta_{[k,F_{n-1})})=c_k+c_{F_{n}+k}+c_{F_{n-1}+k}.$$
\end{proof}

There are many patterns in the crushed array.   Two such patterns can be proven by induction based on the previous theorem.  
\begin{corollary}$c_{F_n}+c_{F_{n-1}+2}=c_{F_n+1}$ and
$c_{F_n+1}=c_{F_{n+1}+2}$ for all $n$.  \end{corollary}

We have many other questions or \textit{ apparent } properties,  all waiting for a proof (though, of course, of varying difficulty).

\begin{itemize}
\item Five inequalities:  $c_{\sigma_2(n)+1}\ge c_{\lfloor n\phi^2\rfloor}\ge c_{\lfloor n\phi\rfloor}\ge c_{\sigma(n)}\ge c_n\ge 0.$
\item The minimum of each row in the crushed array is the leftmost element.  (If true, then the last inequality above, $c_n\ge 0$, is true).  
\item If $c_n\ge 0$ for all $n$, then what do these numbers count?
\item The following sequences have crushed arrays with columns satisfying $x_{n+1}=x_{n}+x_{n-1}-x_{n-j}$ for given $j$:
$$\begin{aligned} & \{s_F(n)\} \text{ has } j=1,\\
& \{R_n\} \text{ has } j=2,\\
& \{c_n\} \text{ has } j=3.
\end{aligned}$$
Is there a general principle at work in this progression?   Is there a similarly defined sequence with $j=4$ for example?
\end{itemize}

\medskip

\noindent MSC2010: 11B83

\end{document}